\date{October 11, 2009}

\documentclass[11pt]{amsart}
\usepackage{latexsym,amsmath,amsfonts,amscd,amssymb}
\usepackage{graphics}
\newtheorem{theorem}{Theorem}[section]
\newtheorem{lemma}[theorem]{Lemma}
\newtheorem{proposition}[theorem]{Proposition}
\newtheorem{corollary}[theorem]{Corollary}

\newtheorem{definition}[theorem]{Definition}

\theoremstyle{remark}
\newtheorem{remark}[theorem]{Remark}

\newcommand{\la}{\langle}
\newcommand{\ra}{\rangle}
\newcommand{\too}{\longrightarrow}

\newcommand{\inc}{\hookrightarrow}
\newcommand{\bd}{\partial}
\newcommand{\x}{\times}
\newcommand{\ox}{\otimes}

\newcommand{\Vol}{\operatorname{Vol}}

\newcommand{\Hol}{\operatorname{Hol}}

\newcommand{\cA}{{\mathcal A}}

\newcommand{\cC}{{\mathcal C}}

\newcommand{\cM}{{\mathcal M}}

\newcommand{\cT}{{\mathcal T}}

\newcommand{\RR}{{\mathbb R}}

\newcommand{\ZZ}{{\mathbb Z}}

\renewcommand{\b}{\beta}

\newcommand{\g}{\gamma}

\begin{document}

\title{Intersection theory for ergodic solenoids}


\author[V. Mu\~{n}oz]{Vicente Mu\~{n}oz}
\address{Instituto de Ciencias Matem\'aticas
CSIC-UAM-UC3M-UCM, Consejo Superior de Investigaciones Cient\'{\i}ficas,
Serrano 113 bis, 28006 Madrid, Spain}

\address{Facultad de
Matem\'aticas, Universidad Complutense de Madrid, Plaza de Ciencias
3, 28040 Madrid, Spain}

\email{vicente.munoz@imaff.cfmac.csic.es}

\author[R. P\'{e}rez Marco]{Ricardo P\'{e}rez Marco}
\address{CNRS, LAGA UMR 7539, Universit\'e Paris XIII, 
99, Avenue J.-B. Cl\'ement, 93430-Villetaneuse, France}


\email{ricardo@math.univ-paris13.fr}

\thanks{Partially supported through Spanish MEC grant MTM2007-63582.}

\maketitle

\begin{abstract}
We develop the intersection theory associated to immersed, oriented
and measured solenoids, which were introduced in \cite{MPM1}.
\end{abstract}

\section{Introduction} \label{sec:introduction}

In \cite{MPM1}, the authors define the concept of $k$-solenoid as an
abstract laminated space, and prove that an solenoid with a transversal measure
immersed in a smooth manifold defines a generalized Ruelle-Sullivan current.
The purpose of the current paper is to study the intersection theory of such
objects.

If $M$ is a smooth manifold, any closed oriented submanifold $N\subset M$
of dimension $k$ determines a homology class in
$H_k(M, \ZZ)$. This homology class in $H_k(M,\RR)$, as dual of De
Rham cohomology, is explicitly given by integration of the
restriction to $N$ of differential $k$-forms on $M$.
For representing real homology classes, we need to consider more
general objects. In \cite{MPM1}, we
define a $k$-solenoid to be a Hausdorff compact space foliated by
$k$-dimensional leaves with finite dimensional transversal structure
(see the precise definition in section \ref{sec:minimal}). For these
oriented solenoids we can consider $k$-forms that we can integrate
provided that we are given a transversal measure invariant by the
holonomy group. We define an immersion of a solenoid $S$ into $M$ to
be a regular map $f: S\to M$ that is an immersion on each leaf. If
the solenoid $S$ is endowed with a transversal measure $\mu$, then
any smooth $k$-form in $M$ can be pulled back to $S$ by $f$ and
integrated. This defines a closed
current that we denote by $(f,S_\mu )$ and
call a generalized current. This defines a homology class
$[f,S_\mu] \in H_k(M,\RR )$. This construction generalizes the currents
introduced by Ruelle and Sullivan in \cite{RS}.

In \cite{MPM4}, we prove that
every real homology class in $H_k(M,\RR )$ can be realized by a
generalized current $(f,S_\mu)$ where $S_\mu$ is an oriented,
minimal, uniquely ergodic solenoid.
Uniquely ergodic solenoids are defined in \cite{MPM1}. These are minimal
solenoids which possess a unique 
transversal measure.
The importance of such solenoids stems from the fact that their 
topology, more precisely the recurrence of any leaf, determines its solenoidal structure.
This is formulated
in a precise way through the Schwartzman measures in \cite{MPM2}.
Finally, in \cite{MPM5} we prove that the generalized currents associated
to oriented, minimal, uniquely ergodic, immersed solenoids are dense in
the space of closed currents.

Let us review the contents of the paper. In section \ref{sec:minimal} we recall
the definition of solenoids and of the generalized current associated to an
immersed, oriented and measured solenoid. In section \ref{sec:homotopy} we prove
that the generalized current is invariant by perturbation (homotopy) of a solenoid.
Section \ref{sec:intersection} is devoted to defining the intersection of two
solenoids $S_{1,\mu_1}$, $S_{2,\mu_2}$
which intersect transversally (i.e.\ when the leaves intersect transversally).
In this case, the intersection is a solenoid and it has a natural transversal measure
associated to the measures of the two given solenoids. The generalized current
of the intersection solenoid is the product of the generalized currents of the given
solenoids. We also prove that we may homotop a solenoid to make it intersect transversally
a submanifold whenever the transversal structure is Cantor.

In general, it is not possible to perturb solenoids to make them intersect transversally.
This clearly holds in the case of solenoids. But it is also true when we try to perturb
smoothly solenoids with transversal Cantor structure, as the 
persistence of homoclinic tangencies for stable and unstable foliations shows (see \cite{Takens}).
So the concept of almost everywhere transversality introduced in section  \ref{sec:aet} is
very useful. In the case of complementary dimensions, two solenoids are said to intersect
almost everywhere transversally if
$\mu_1\times \mu_2$-almost all leaves intersect transversally
and the other leaves intersect in isolated points. We define a measure for the intersection
and the integral of the measure (taking also into account the intersection index) equals
the product of the generalized currents. In the case of non-complementary dimensions, we shall
require that the models are conjugated to analytic solenoids. This allows to define an
intersection current supported on the intersection of the solenoids. That is worked out
in section \ref{sec:analytic}.

\bigskip

\noindent \textbf{Acknowledgements.} \
The authors are grateful to Alberto Candel, Etienne Ghys, Nessim Sibony,
Dennis Sullivan and Jaume Amor\'os
for their comments and interest on this work.
The first author wishes to acknowledge Universidad Complutense de
Madrid and Institute for Advanced Study at Princeton for their
hospitality and for providing excellent working conditions.  The second author
thanks Jean Bourgain and the IAS at Princeton for their hospitality
and facilitating the collaboration of both authors.

\section{Measured solenoids and generalized currents} \label{sec:minimal}

Let us recall the definition of a $k$-solenoid from \cite{MPM1}.

\begin{definition}\label{def:W}
  Let $0\leq s,r \leq \omega$, $r\geq s$, and let $k,l\geq 0$ be two integers.
  A foliated manifold (of dimension $k+l$, with $k$-dimensional
  leaves, of regularity $C^{r,s}$) is a smooth
  manifold $W$ of dimension $k+l$ endowed with an atlas
  $\cA=\{ (U_i,\varphi_i)\}$, $\varphi_i:U_i\to \RR^k\x \RR^l$,
  whose changes of charts
 $$
 \varphi_{ij}=\varphi_i\circ \varphi_j^{-1}: \varphi_j(U_i\cap U_j)
 \to \varphi_i(U_i\cap U_j) \ \, ,
 $$
are of the form $\varphi_{ij}(x,y)=(X_{ij}(x,y), Y_{ij}(y))$,
where $Y_{ij}(y)$ is of class $C^s$ and $X_{ij}(x,y)$ is of class
$C^{r,s}$.

A flow-box for $W$ is a pair $(U,\varphi)$ consisting of an open
subset $U\subset W$ and a map $\varphi:U\to \RR^k\x \RR^l$ such that
$\cA \cup \{(U,\varphi)\}$ is still an atlas for $W$.
\end{definition}

(Here $C^\omega$ are the analytic functions.)


Given two foliated manifolds $W_1$, $W_2$ of dimension $k+l$, with
$k$-dimensional leaves, and of regularity $C^{r,s}$, a regular map
$f:W_1\to W_2$ is a continuous map which is locally, in
flow-boxes, of the form $f(x,y)=(X(x,y),Y(y))$, where $Y$ is of
class $C^{s}$ and $X$ is of class $C^{r,s}$.
A diffeomorphism $\phi:W_1\to W_2$ is a homeomorphism such that
$\phi$ and $\phi^{-1}$ are both regular maps.

\begin{definition}\label{def:k-solenoid}\textbf{\em ($k$-solenoid)}
Let  $0\leq r \leq s\leq \omega$, and let $k,l\geq 0$ be two
integers. A pre-solenoid of dimension $k$, of class $C^{r,s}$ and
transversal dimension $l$ is a pair $(S,W)$ where $W$ is a
foliated manifold and $S\subset W$ is a compact subspace which is
a collection of leaves.

Two pre-solenoids $(S,W_1)$ and $(S,W_2)$ are equivalent if there
are open subsets $U_1\subset W_1$, $U_2\subset W_2$ with $S\subset
U_1$ and $S\subset U_2$, and a diffeomorphism $f: U_1\to U_2$ such
that $f$ is the identity on $S$.

A $k$-solenoid of class $C^{r,s}$ and transversal dimension $l$
(or just a $k$-solenoid, or a solenoid) is an equivalence class of
pre-solenoids.
\end{definition}

We usually denote a solenoid by $S$, without making explicit
mention of $W$. We shall say that $W$ defines the solenoid
structure of $S$.

\begin{definition}\label{def:flow-box}
\textbf{\em (Flow-box)} Let $S$ be a solenoid. A flow-box for $S$
is a pair $(U,\varphi)$ formed by an open subset $U\subset S$ and
a homeomorphism
 $$
 \varphi : U\to D^k\times K(U) \, ,
 $$
where $D^k$ is the $k$-dimensional open ball and $K(U)\subset
\RR^l$, such that there exists a foliated manifold $W$ defining
the solenoid structure of $S$, $S\subset W$, and a flow-box
$\hat\varphi:\hat U\to \RR^k\x\RR^l$ for $W$, with $U=\hat U\cap
S$,  $\hat\varphi(U)= D^k\x K(U)\subset \RR^k\x\RR^l$ and
$\varphi=\hat\varphi_{|U}$.

The set $K(U)$ is the transversal space of the flow-box. The
dimension $l$ is the transversal dimension.
\end{definition}

As $S$ is locally compact, any point of $S$ is contained in a flow-box $U$ whose closure
$\overline U$ is contained in a bigger flow-box. For such flow-box, $\overline U
\cong \overline{D}^k \times \overline K(U)$, where $\overline D^k$
is the closed unit ball, $\overline K(U)$ is some compact subspace of
$\RR^l$, and $U=D^k \times  K(U) \subset
\overline{D}^k \times \overline K(U)$. 
All flow-boxes that we shall use are of this type without making further
explicit mention. 

\begin{definition} \label{def:leaf}\textbf{\em (Leaf)}
A leaf of a $k$-solenoid $S$ is a leaf $l$ of any foliated manifold
$W$ inducing the solenoid structure of $S$, such that $l\subset S$.
Note that this notion is independent of $W$.
\end{definition}

\begin{definition} \label{def:oriented-solenoid}\textbf{\em (Oriented solenoid)}
An oriented solenoid is a solenoid $S$ such that there is a
foliated manifold $W\supset S$ inducing the solenoid structure of
$S$, where $W$ has oriented leaves (in a transversally continuous way).
\end{definition}

A solenoid is minimal if it does not contain a proper sub-solenoid. In
\cite{MPM1} it is proven that minimal solenoids always exist.

\begin{definition}\label{def:transversal}\textbf{\em (Transversal)}
Let $S$ be a $k$-solenoid. A local transversal at a point $p\in S$
is a subset $T$ of $S$ with $p\in T$, such that there is a flow-box $(U, \varphi)$
of $S$ with $U$ a neighborhood of $p$ containing $T$ and such that
 $$
 \varphi (T)=\{0\}\x K(U) \, .
 $$

A transversal $T$ of $S$ is a compact subset of $S$ such that for
each $p\in T$ there is an open neighborhood $V$ of $p$ such that
$V\cap T$ is a local transversal at $p$.
\end{definition}

If $S$ is a $k$-solenoid of class $C^{r,s}$, then any transversal
$T$ inherits an  $l$-dimensional $C^s$-Whitney structure.

\begin{definition}\label{def:global-transversal}
A transversal $T$ of $S$ is a global transversal if all leaves
intersect $T$.
\end{definition}

\begin{definition}\label{def:holonomy}\textbf{\em (Holonomy)}
Given two points $p_1$ and $p_2$ in the same leaf, two local
transversals $T_1$ and $T_2$, at $p_1$ and $p_2$ respectively, and
a path $\gamma :[0,1]\to S$, contained in the leaf with endpoints
$\g(0)=p_1$ and $\g(1)=p_2$, we define a germ of a map (the
holonomy map)
 $$
 h_\gamma: (T_1,p_1) \to (T_2, p_2)\, ,
 $$
by lifting $\gamma$ to nearby leaves.

We denote by ${\Hol}_S (T_1, T_2)$ the set of germs of holonomy
maps from $T_1$ to $T_2$. These form the holonomy pseudo-group.
\end{definition}

\begin{definition} \label{def:transversal-measure} \textbf{\em (Transversal measure)}
Let $S$ be a $k$-solenoid. A transversal measure $\mu=(\mu_T)$ for
$S$ associates to any local transversal $T$ a locally finite measure
$\mu_T$ supported on $T$, which are invariant by the holonomy
pseudogroup. More precisely, if $T_1$ and $T_2$ are two transversals
and $h : V\subset T_1 \to T_2$ is a holonomy map, then
 $$
 h_* (\mu_{T_1}|_{V})= \mu_{T_2}|_{h(V)} \, .
 $$
We assume that a transversal measure $\mu$ is non-trivial, i.e. for
some $T$, $\mu_T$ is non-zero.

We denote by $S_\mu$ a $k$-solenoid $S$ endowed with a transversal
measure $\mu=(\mu_T)$. We refer to $S_\mu$ as a measured solenoid.
\end{definition}

\bigskip

We fix now a $C^\infty$ manifold $M$ of dimension $n$.

\begin{definition}\label{def:solenoid-in-manifold}
\textbf{\em (Immersion and embedding of solenoids)} Let $S$ be a
$k$-solenoid of class $C^{r,s}$ with $r\geq 1$. An immersion
 $$
 f:S\to M
 $$
is a regular map (that is, it has an extension $\hat{f}:W\to M$ of
class $C^{r,s}$, where $W$ is a foliated manifold which defines
the solenoid structure of $S$), such that the differential
restricted to the tangent spaces of leaves has rank $k$ at every
point of $S$. We say that $f:S\to M$ is an immersed solenoid.


Let $r,s\geq 1$. A transversally immersed solenoid $f:S\to M$ is a
regular map $f:S\to M$ such that it admits an extension $\hat{f}:W\to M$ which is
an immersion (of a $(k+l)$-dimensional manifold into an
$n$-dimensional one) of class $C^{r,s}$, such that the images of
the leaves intersect transversally in $M$.

An embedded solenoid $f:S\to M$ is a transversally immersed solenoid of class $C^{r,s}$,
with $r,s\geq 1$,
with injective $f$, that is, the leaves do not intersect or
self-intersect.
\end{definition}

Note that under a transversal immersion, resp.\ an embedding,
$f:S\to M$, the images of the leaves are immersed, resp.\
injectively immersed, submanifolds.


Let $M$ be a smooth manifold. We shall denote the space of compactly supported currents
of dimension $k$ by
 $$
 \cC_k(M) \, .
 $$
These currents are functionals $T:\Omega^k(M)\to \RR$. 
A current $T\in \cC_k(M)$ is closed if $T(d\alpha)=0$
for any $\alpha\in \Omega^{k-1}(M)$. Therefore, by restricting to the closed forms,
a closed current $T$ defines a linear map
  $$
 [T]: H^k(M,\RR) \longrightarrow \RR\, .
 $$
By duality, $T$ defines a
real homology class $[T]\in H_k(M,\RR)$.

\begin{definition}\label{def:Ruelle-Sullivan}\textbf{\em (Generalized currents)}
Let $S$ be an oriented  $k$-solenoid of class $C^{r,s}$, $r\geq
1$, endowed with a transversal measure $\mu=(\mu_T)$. An immersion
 $$
 f:S\to M
 $$
defines a current $(f,S_\mu)\in \cC_k(M)$, called generalized Ruelle-Sullivan
current (or just generalized current), as follows.
Let $\omega$ be an $k$-differential form in $M$. The pull-back
$f^* \omega$ defines a $k$-differential form on the leaves of $S$.

Let $S=\bigcup_i U_i$ be an open cover of the solenoid.
Take a partition of unity $\{\rho_i\}$ subordinated to
the covering $\{U_i\}$. We define
  $$
   \la (f,S_\mu),\omega \ra = \sum_i
   \int_{K(U_i)} \left( \int_{L_y} \rho_i f^* \omega \right)
   d\mu_{K(U_i)} (y) \, ,
  $$
where $L_y$ denotes the horizontal disk of the flow-box.

The current $(f,S_\mu)$ is closed, hence it defines a
real homology class
 $$
 [f,S_\mu]\in H_k(M,\RR )\, ,
 $$
called Ruelle-Sullivan homology class.
\end{definition}

{}From now on, we shall consider a $C^\infty$ compact and oriented
manifold $M$ of dimension $n$. Let $f:S_\mu\to M$ be an oriented
measured $k$-solenoid immersed in $M$.
We shall denote
  $$
  [f,S_\mu]^* \in H^{n-k}(M,\RR)\, ,
  $$
the dual of $[f,S_\mu]$ under the Poincar\'{e} duality isomorphism $H_k(M,\RR)\cong
H^{n-k}(M,\RR)$.

%

\section{Homotopy of solenoids}\label{sec:homotopy}

Let us see that the Ruelle-Sullivan homology class defined by an immersed
oriented measured $k$-solenoid does not change by perturbations.

\begin{definition}\textbf{\em (Solenoid with boundary)}
  Let $0\leq s \leq r\leq \omega$, and let $k,l\geq 0$ be two integers.
  A foliated manifold with boundary (of dimension $k+l$, with $k$-dimensional
  leaves, of class $C^{r,s}$) is a smooth
  manifold $W$ with boundary, of dimension $k+l$, endowed with an
  atlas
  $\{ (U_i,\varphi_i)\}$ of charts
  $$
  \varphi_i: U_i \to \varphi_i(U_i) \subset \RR^{k+l}_+ =\{(x_1,\ldots,
  x_k, y_1,\ldots, y_l \ ; \ x_1\geq 0\} \, ,
  $$
whose changes of charts are of the form $\varphi_i\circ
\varphi_j^{-1}(x,y) = (X_{ij}(x,y), Y_{ij}(y))$, where $Y_{ij}(y)$
is of class $C^s$ and $X_{ij}(x,y)$ is of class $C^{r,s}$.

A pre-solenoid with boundary is a pair $(S,W)$ where $W$ is a
foliated manifold with boundary and $S\subset W$ is a compact
subspace which is a collection of leaves.

Two pre-solenoids with boundary $(S,W_1)$ and $(S,W_2)$ are
equivalent if there are open subsets $U_1\subset W_1$, $U_2\subset
W_2$ with $S\subset U_1$ and $S\subset U_2$, and a diffeomorphism
$f: U_1\to U_2$ (preserving leaves, of class $C^{r,s}$) which is the
identity on $S$.

A $k$-solenoid with boundary is an equivalence class of
pre-solenoids with boundary.
\end{definition}

Note that any manifold with boundary is a solenoid with boundary.

The boundary of a $k$-solenoid with boundary $S$ is the
$(k-1)$-solenoid (without boundary) $\bd S$ defined by the foliated
manifold $\bd W$, where $W$ is a foliated manifold with boundary
defining the solenoid structure of $S$.

A $k$-solenoid with boundary $S$ has two types of flow-boxes. If
$p\in S-\bd S$ is an interior point, then there is a flow-box
$(U,\varphi)$ with $p\in U$, of the form $\varphi : U\to D^k\times
K(U)$. If $p\in \bd S$ is a boundary point, then there is a flow-box
$(U,\varphi)$ with $p\in U$ such that $\varphi$ is a homeomorphism
 $$
 \varphi : U\to D^k_+\times K(U) \, ,
 $$
where $D^k_+=\{ (x_1,\ldots, x_k)\in D_k \, ; \, x_1\geq 0\}$, and
$K(U)\subset \RR^l$, $\varphi(p)=(0,\ldots, 0, y_0)$, for some
$y_0\in K$. Note that writing
 $$
 U'=\bd S\cap U =\varphi^{-1}(D^{k-1}\x K(U))\, ,
 $$
where $D^{k-1} = \{ (0,x_2,\ldots, x_k)\in D_k \} \subset D_k^+$, we have that
$(U',\varphi_{|U'})$ is a flow-box for $\bd S$. Therefore, if $T$ is
a transversal for $\bd S$, then it is also transversal for $S$.

For a solenoid with boundary $S$ there is also a well-defined notion
of holonomy pseudo-group. If $T$ is a local transversal for $\bd S$,
and $h:T\to T$ is a holonomy map for $\bd S$ defined by a path in
$\bd S$, then $h$ lies in the holonomy pseudo-group of $S$. So
  $$
  \Hol_{\bd S}(T)\subset \Hol_S(T)\, ,
  $$
but they are in general not equal. In particular, if $S$ is
connected and minimal with non-empty boundary then
 $$
 \cM_\cT(S) \subset \cM_\cT(\bd S) \, .
 $$
That is, if $\mu=(\mu_T)$ is a transversal measure for $S$, then it
yields a transversal measure for $\bd S$, by considering only those
transversals $T$ which are transversals for $\bd S$. We denote this
transversal measure by $\mu$ again.

If $S$ comes equipped with an orientation, then $\bd S$ has a
natural induced orientation.
Note that any leaf $l\subset S$ is a manifold with boundary and
each connected component of $\bd l$ is a leaf of $\bd S$.

\begin{theorem}\textbf{\em (Stokes theorem)} \label{thm:Stokes}
  Let $f:S_\mu\to M$ be an oriented $(k+1)$-solenoid with boundary,
  endowed with a transversal measure, and immersed into a smooth
  manifold $M$. Let $\omega$ be a $k$-form on $M$. Then
  $$
  \la [f,S_\mu], d\omega \ra = \la [f_{|\bd S},\bd S_\mu] , \omega
  \ra \, .
  $$
\end{theorem}

\begin{proof}
  Let $\{U_i\}$ be a covering of $S$ by flow-boxes, and let $\{\rho_i\}$
  be a partition of unity subordinated to it. Adding up the
  equalities
  $$
  \begin{aligned}
    \int_{K(U_i)}& \left( \int_{L_y} d\rho_i \wedge f^* \omega \right)
    d\mu_{K(U_i)} (y)
  +
    \int_{K(U_i)} \left( \int_{L_y} \rho_i f^* d\omega \right)
    d\mu_{K(U_i)} (y) \\
  &=
    \int_{K(U_i)} \left( \int_{L_y} d (\rho_i f^* \omega) \right)
    d\mu_{K(U_i)} (y)
  =
    \int_{K(U_i)} \left( \int_{\bd L_y} \rho_i f^* \omega \right)
    d\mu_{K(U_i)} (y)\,,
    \end{aligned}
  $$
 for all $i$, and using that $\sum d\rho_i \equiv 0$, we get
  $$
  \begin{aligned}
  \la  [f,S_\mu], d\omega \ra &\, =
   \sum_i \int_{K(U_i)} \left( \int_{L_y} \rho_i f^* d\omega \right)
    d\mu_{K(U_i)} (y) \\
  & \, =
   \sum_i \int_{K(U_i)} \left( \int_{\bd L_y} \rho_i f^* \omega \right)
    d\mu_{K(U_i)} (y)
    =
    \la [f_{|\bd S},\bd S_\mu] , \omega  \ra \, .
  \end{aligned}
  $$
\end{proof}

Let $S$ be a $k$-solenoid of class $C^{r,s}$. We give $S\x I=S\x
[0,1]$ a natural $(k+1)$-solenoid with boundary
structure of the same class, by taking
a foliated manifold $W\supset S$ defining the solenoid structure of
$S$, and foliating $W\x I$ with the leaves $l\x I$,
$l\subset W$ being a leaf of $W$. Then $S\x I \subset W\x
I$ is a $(k+1)$-solenoid with boundary.
The boundary of $S\x I$ is
 $$
 (S\x\{0\}) \sqcup (S\x \{1\})\, .
 $$

If $S$ is oriented then $S\x I$ is naturally oriented and its boundary
consists of  $S\x\{0\} \cong S$ with orientation reversed, and
$S\x\{1\} \cong S$ with orientation preserved.

Moreover if $T$ is a transversal for $S$, then $T'=T \x \{0\}$
is a transversal for $S'=S\x I$.
The following is immediate.

\begin{lemma} \label{lem:SxI}
There is an identification of the holonomies of $S$ and $S\x I$.
More precisely, under the identification $T\cong T' =T \x \{0\}$,
 $$
 \Hol_S(T) =\Hol_{S\x I}(T')\, .
 $$

In particular,
 $$
 \cM_\cT(S)=\cM_\cT(S\x I)\, .
 $$
\end{lemma}

\begin{definition}\textbf{\em (Equivalence of immersions)}\label{def:14.equivalence}
Two solenoid immersions $f_0:S_0\to M$ and $f_1 :S_1\to M$ of class
$C^{r,s}$ in $M$ are immersed equivalent if there is a
$C^{r,s}$-diffeomorphism $h:S_0\to S_1$ such that
 $$
 f_0=f_1\circ h \, .
 $$
Two measured solenoid immersions are immersed equivalent if $h$ can be chosen
to preserve the transversal measures.
\end{definition}

\begin{definition}\textbf{\em (Homotopy of immersions)}\label{def:14.homotopy}
Let $S$ be a $k$-solenoid of class $C^{r,s}$ with $r\geq 1$. A
homotopy between immersions $f_0:S\to M$ and $f_1: S\to M$ is an
immersion of solenoids $f:S\x I\to M$ such that $f_0(x)=f(x,0)$ and
$f_1(x)=f(x,1)$.
\end{definition}

\begin{definition} \textbf{\em (Cobordism of solenoids)}
\label{def:14.cobordism} Let $S_0$ and $S_1$ be two
$C^{r,s}$-solenoids. A cobordism of solenoids is a
$(k+1)$-solenoid  $S$  with boundary $\bd S=S_0\sqcup S_1$.

If $S_0$ and $S_1$ are oriented, then an oriented cobordism is a
cobordism $S$ which is an oriented solenoid such that it induces the
given orientation on $S_1$ and the reversed orientation on $S_0$.

If $S_0$ and $S_1$ have transversal measures $\mu_0$ and $\mu_1$,
respectively, then a measured cobordism is a cobordism $S$ endowed
with a transversal measure $\mu$ inducing the measures $\mu_0$ and
$\mu_1$ on $S_0$ and $S_1$, respectively.
\end{definition}

\begin{definition}\textbf{\em (Homology equivalence)}
\label{def:homology}
 Let $f_0:S\to M$ and $f_1: S\to M$ be two immersed solenoids in $M$.
 We say that they are homology equivalent if there exists a
 cobordism of solenoids $S$ between $S_0$ and $S_1$ and a solenoid
 immersion $f:S \to M$ with $f_{|S_0}=f_0$, $f_{|S_1}=f_1$.
 We call $f:S \to M$ a homology between $f_0:S\to M$ and $f_1: S\to M$.

 Let $f_0:S_{0,\mu_0}\to M$ and $f_1:S_{1,\mu_1}\to M$ be two immersed
 oriented measured solenoids. They are homology equivalent if
 there exists an immersed oriented measured solenoid $f:S_\mu\to M$ such that
 $f:S\to M$ is a homology between $f_0:S\to M$ and $f_1: S\to M$ and
 $S_\mu$ is a measured oriented cobordism from $S_0$ to $S_1$.
\end{definition}

Clearly two homotopic immersions of a solenoid give homology
equivalent immersions.

\begin{theorem} \label{thm:homeo-RS}
 Suppose that two oriented measured solenoids
 $f_0:S_{0,\mu_0}\to M$ and $f_1:S_{1,\mu_1}\to M$  immersed in $M$
 are homology equivalent. Then the
generalized currents coincide
   $$
   [f_0,S_{0,\mu_0}]=[f_1,S_{1,\mu_1}]\, .
   $$

 The same happens if they are immersed equivalent.
\end{theorem}

\begin{proof}
In the first case, let $\omega$ be a closed $k$-form on $M$, then
Stokes' theorem gives
  $$
  \la [f_1,S_{1,\mu_1}] , \omega \ra- \la [f_0,S_{0,\mu_0}] ,
  \omega \ra =
  \la [f_{|\bd S},\bd S_\mu] , \omega \ra = \la [f,S_\mu],d\omega\ra =0\, .
  $$

In the second case, $f_0=f_1\circ h$ implies that the actions of the
generalized currents over a closed form on $M$ coincide, since the
pull-back of the form to the solenoids agree through the
diffeomorphism $h$, and the integrals over the transversal measure
gives the same numbers, since the measures correspond by $h$.
\end{proof}

\begin{remark} \label{rem:que-numero-le-pongo}
 In both definitions \ref{def:14.homotopy} and \ref{def:14.cobordism},
 we do not need to require that $f$ be an
 immersion. Actually, the generalized current $[f,S_\mu]$ makes
 sense for any measured solenoid $S_\mu$ and any regular map $f:S\to
 M$, of class $C^{r,s}$ with $r\geq 1$. Theorem \ref{thm:homeo-RS}
 still holds with these extended notions.
\end{remark}

\section{Intersection theory of solenoids}\label{sec:intersection}

Let $M$ be a smooth $C^\infty$ oriented manifold.

\begin{definition} \textbf{\em (Transverse intersection)}
\label{def:15.1} Let $f_1:S_1\to M$, $f_2:S_2\to M$ be two immersed
solenoids in $M$. We say that they intersect transversally if, for
every $p_1\in S_1$, $p_2\in S_2$ such that $f_1(p_1)=f_2(p_2)$, the
images of the leaves through $p_1$ and $p_2$ intersect
transversally.
\end{definition}

If two immersed solenoids $f_1:S_1\to M$, $f_2:S_2\to M$, of dimensions
$k_1$, $k_2$ respectively, intersect transversally, we define the
intersection solenoid $f:S\to M$ as follows. The solenoid $S$ is:
  \begin{equation} \label{eqn:intersect-solenoids}
   S=\{ (p_1,p_2) \in S_1\x S_2 \ ; \ f_1(p_1)=f_2(p_2)\} \, .
   \end{equation}
and the map  $f:S\to M$ is given by
  \begin{equation} \label{eqn:immersion-intersect}
 f(p_1,p_2) =f_1(p_1)=f_2(p_2) , \quad (p_1,p_2) \in S\, .
   \end{equation}
We will see that $S$, the intersection solenoid, is indeed a
solenoid. Also the intersection $f:S\to M$ of the two immersed
solenoids $f_1:S_1\to M$, $f_2:S_2\to M$ is an immersed solenoid. In order
to prove this, we consider the intersection of the product solenoid
$F=f_1\x f_2:S_1\x S_2\to M\x M$ with the diagonal $\Delta \subset
M\x M$. So we have to analyze first the case of the intersection of
an immersed solenoid with a submanifold. The notion of transverse
intersection given in definition \ref{def:15.1} applies to this case
(a submanifold is an embedded solenoid).

\begin{lemma} \label{lem:intersect-submfd}
  Let $f:S\to M$ be an immersed $k$-solenoid in $M$ intersecting
  transversally an embedded closed submanifold $N\subset M$ of
  codimension $q$. Suppose that $S'=f^{-1}(N)\subset S$ is
  non-empty, then $f_{|S'}:S' \to M$ is an immersed $(k-q)$-solenoid in $N$.

  If $S$ and $N$ are oriented, so is $S'$.

  If $S$ has a transversal measure $\mu$, then $S'$ inherits
  a natural transversal measure, also denoted by $\mu$.
\end{lemma}

\begin{proof}
First of all, note that $S'$ is a compact and Hausdorff space.

Let $W$ be a foliated manifold defining the solenoid structure of
$S$ such that there is a smooth map $\hat f:W\to M$ of class
$C^{r,s}$, extending $f:S\to M$, which is an immersion on leaves. By
definition, for any leaf $l\subset S$, $f(l)$ is transverse to $N$.
Thus reducing $W$ if necessary, the same transversality property
occurs for any leaf of $W$. The transversality of the leaves implies
that the map $\hat f:W\to M$ is transversal to the submanifold
$N\subset M$, meaning that for any $p\in W$ such that $\hat f(p)\in
N$,
 $$
 d\hat f(p)(T_pW) + T_{\hat f(p)}N =T_{\hat f(p)}M \, .
 $$
This implies that $W'=\hat f^{-1}(N)$ is a submanifold of $W$ of
codimension $q$ (in particular, $k-q\geq 0$). Moreover, it is
foliated by the connected components $l'$ of $l\cap \hat
f^{-1}(N)=(\hat f_{|l})^{-1}(N)$, where $l$ are the leaves of $W$.
By transversality of $\hat f$ along the leaves, $l'$ is a
$(k-q)$-dimensional submanifold of $l$. So $W'$ is a foliated
manifold with leaves of dimension $k-q$. This gives the required
solenoid structure to $S'=S \cap \hat f^{-1}(N)=f^{-1}(N)$.

Clearly, $f_{|S'}:S'\to N$ is an immersion (of class $C^{r,s}$)
since $\hat f_{|W'}:W'\to N$ is a smooth map which is an immersion
on leaves.

If $S$ and $N$ are oriented, then each intersection $l'=l\cap
\hat f^{-1}(N)$ is also oriented (using that $M$ is oriented as well). Therefore the
leaves of $S'$ are oriented, and hence $S'$ is an oriented solenoid.

  Let $p\in S'$ and let $U\cong D^k\x K(U)$ be a flow-box for $S$ around $p$.
  We can take $U$ small enough so that $f(U)$ is contained in a
  chart of $M$ in which $N$ is defined by functions $x_1=\ldots =x_q
  =0$. By the transversality property, the differentials
  $dx_1,\ldots, dx_q$ are linearly independent on each leaf $f(D^k\x
  \{y\})$, $y\in K(U)$. Therefore, $x_1, \ldots, x_q$ can be completed to a set of
  functions $x_1,\ldots, x_k$ such that $dx_1,\ldots, dx_k$ are a
  basis of the cotangent space for each leaf (reducing $U$ if
  necessary). Thus the pull-back of $x=(x_1,\ldots,x_k)$ to $U$
  give coordinates functions so that, using the coordinate $y\in K(U)$
  for the transversal direction, $(x,y)$ are coordinates for $U$,
  and $f^{-1}(N)$ is defined as $x_1=\ldots =x_q
  =0$. This means that
    $$
    S'\cap U \cong \{ (0,\ldots, 0, x_{q+1},\ldots, x_k,y)\in D^k\x
    K(U)\} \cong D^{k-q} \x K(U)\, .
    $$
Therefore any local transversal $T$ for $S'$ is a local transversal
for $S$, and any holonomy map for $S'$ is a holonomy map for $S$. So
  $$
  \Hol_{S'}(T) \subset \Hol_{S}(T) \, .
  $$
Hence a transversal measure for $S$ gives a transversal measure
for $S'$.
\end{proof}

Now we can address the general case.

\begin{proposition} \label{prop:15.3}
Suppose that $f_1:S_1\to M$, $f_2:S_2\to M$ are two immersed solenoids in $M$
intersecting transversally, and let $S$ be its intersection solenoid
defined in (\ref{eqn:intersect-solenoids}) and let $f$ be the map
(\ref{eqn:immersion-intersect}). If $S\neq \emptyset$,
then $f:S\to M$ is an immersed solenoid of dimension
$k=k_1+k_2-n$ (in particular, $k$ is a non-negative number).

If $S_1$ and $S_2$ are both oriented, then $S$ is also oriented.

If $S_1$ and $S_2$ are endowed with transversal measures $\mu_1$
and $\mu_2$ respectively,
then $S$ has an induced measure $\mu$.
\end{proposition}

\begin{proof}
  The product $S_1\x S_2$ is a
  $(k_1+k_2)$-solenoid and
   $$
   F=f_1\x f_2 :S_1\x S_2 \to M\x M
   $$
  is an immersion. Let $\Delta\subset M\x M$ be the diagonal. There is
  an identification (as sets)
   $$
   S = (S_1\x S_2) \cap F^{-1}(\Delta)\, .
   $$
 The condition that $f_1:S_1\to M$, $f_2:S_2\to M$ intersect transversally can
 be translated into that
 $F,S_1\x S_2\to M\x M$ and $\Delta$ intersect transversally in $M\x M$.

Therefore applying lemma \ref{lem:intersect-submfd},
$(S,F_{|S})$ is an immersed $k$-solenoid,
where $F_{|S}:S \to \Delta$ is defined as $F(x_1,x_2)=f_1(x_1)$.
Using the diffeomorphism $M\cong\Delta$, $x\mapsto (x,x)$,
$F_{|S}$  corresponds to $f:S\to M$. So $f:S\to M$ is an immersed $k$-solenoid.

If $S_1$ and $S_2$ are both oriented, then $S_1\x S_2$ is also oriented. By
lemma \ref{lem:intersect-submfd}, $S$ inherits an orientation.

If $S_1$ and $S_2$ are endowed with transversal measures $\mu_1$
and $\mu_2$, then $S_1\x S_2$ has a product transversal measure $\mu$.
For any local transversals $T_1$ and $T_2$ to $S_1$ and $S_2$, respectively,
$T=T_1\x T_2$ is a local transversal to $S_1\x S_2$ (and conversely). We
define
 \begin{equation} \label{eqn:mu-product}
 \mu_T= \mu_{1,T_1} \x \mu_{2,T_2} \, .
 \end{equation}
Now lemma \ref{lem:intersect-submfd} applies to give the transversal measure for $S$.
Note that the local transversals to $S$ are of the form $T_1\x T_2$, for some
local transversals $T_1$ and $T_2$ to $S_1$ and $S_2$.
\end{proof}

\begin{remark} \label{rem:15.4}
If $k_1+k_2=n$ then $S$ is a $0$-solenoid.
For a $0$-solenoid $S$, an orientation is a continuous
assignment $\epsilon:S\to \{\pm 1\}$ of
sign to each point of $S$.

Note also that for a $0$-solenoid $S$, $T=S$ is a transversal and a
transversal measure is a Borel measure on $S$.
\end{remark}

Let $f_1:S_1\to M$, $f_2:S_2\to M$ be two immersed solenoids in $M$
intersecting transversally, with $f:S\to M$ its intersection solenoid.
Let $p=(p_1,p_2)\in S$. Then we can choose flow-boxes $U_1=
D^{k_1}\x K(U_1)$ for $S_1$ around $p_1$ with coordinates
$(x_1,\ldots, x_{k_1}, y)$, and $U_2= D^{k_2}\x K(U_2)$ for $S_2$
around $p_2$ with coordinates $(x_1,\ldots, x_{k_2}, z)$, and
coordinates for $M$ around $f(p)$, such that
  $$
  \begin{aligned}
   f_1(x,y) & \, = (x_1,\ldots,x_{k_1+k_2-n},x_{k_1+k_2-n+1},\ldots,
   x_{k_1}, B_1(x,y),\ldots, B_{n-k_1}(x,y)) \, ,\\
   f_2(x,z) & \, = (x_1,\ldots, x_{k_1+k_2-n}, C_1(x,z),\ldots,
   C_{n-k_2}(x,z), x_{k_1+k_2-n+1},\ldots, x_{k_2})\, .
  \end{aligned}
  $$
Then $S$ is defined
locally as $D^{k_1+k_2-n}\x K(U_1)\x K(U_2)$ with coordinates $(x,y,z)=(x_1,\ldots,
x_{k_1+k_2-n},y,z)$ and
  $$
  \begin{aligned}
  f( & x_1,\ldots,
  x_{k_1+k_2-n},y,z)= \\ &= (x_1,\ldots, x_{k_1+k_2-n},C_1(x,z),\ldots,
  C_{n-k_2}(x,z),B_1(x,y),\ldots, B_{n-k_1}(x,y) ) \, .
  \end{aligned}
  $$

\medskip

\begin{theorem}\label{thm:product-dual-1}
Let $f:S_\mu\to M$ be an oriented measured $k$-solenoid immersed in $M$
intersecting transversally a closed subvariety $i:N\inc M$ of
codimension $q$, such that $S'=f^{-1}(N)\subset S$ is non-empty.
Consider the oriented measured $(k-q)$-solenoid immersed in $N$,
$f':S'\to M$, where $f'=f_{|S'}$. Then, under the restriction map
  \begin{equation} \label{eqn:restr}
  i^*: H^{n-k}_c(M) \to H^{(n-q)-(k-q)}_c(N)\, ,
  \end{equation}
the dual of the Ruelle-Sullivan homology class $[f,S_\mu]^*$ maps to
$[f',S'_\mu]^*$.
\end{theorem}

\begin{proof}
Let $U\subset M$ be a tubular neighbourhood of $N$ with projection
$\pi:U\to N$. Note that $U$ is diffeomorphic to the unit disc bundle
of the normal bundle of $N$ in $M$. Let $\tau$ be a Thom form for
$N\subset M$, that is a closed form $\tau \in \Omega^q(M)$ supported
in $U$, whose integral in any normal space $\pi^{-1}(n)$, $n\in N$,
is one. The dual of the map (\ref{eqn:restr}) under Poincar\'e
duality is the map
 $$
 H^{k-q}(N) \to H^k(M)\, ,
 $$
which sends $[\b]\in H^{k-q}(N)$ to $[\tilde\b]$, where
$\tilde\b=\pi^*\b \wedge \tau$ (this form is extended from $U$ to
the whole of $M$ by zero). So we only need to see that
 $$
 \la [f,S_\mu], \tilde\b\ra = \la [f',S'_\mu], \b\ra\, .
 $$

Take a covering of $S$ by flow-boxes $U_i\cong D^k \x K(U_i)\cong
D^{q}\x D^{k-q}\x K(U_i)$ so that $U_i'=U_i \cap S'$ is given by
$x_1=\ldots =x_q=0$. Making the tubular neighborhood $U\supset N$
smaller if necessary, we can arrange that $f^{-1}(U)\cap U_i$ is
contained in $D^{q}_r\x D^{k-q}\x K(U_i)$, for some $r<1$. It is
easy to construct a map $\tilde\pi: f^{-1}(U) \to f^{-1}(N)$ which
consists on projecting in the normal directions along the leaves.
Then $f\circ \tilde\pi$ and $\pi\circ f$ are homotopic.

Let $S_i'$ be a measurable partition of $S'$ with $S_i'\subset
U_i'$. We may assume that $S_i=\tilde\pi^{-1}(S_i')$ is contained in
$U_i$. 
The sets $S_i$ form a measurable partition containing $f^{-1}(U)$,
the support of $f^*\tilde{\beta}=f^*(\pi^*\b \wedge \tau)$. Then
  $$
  \begin{aligned}
  \la [f,S_\mu], \tilde\b\ra = & \,
  \sum_i \int_{K(U_i)} \left( \int_{S_i\cap (D^k\x \{y\})} f^*(\pi^*\b \wedge \tau) \right)
   d \mu_{K(U_i)}(y) \\
   = &\, \sum_i \int_{K(U_i)} \left( \int_{S_i\cap (D^k\x \{y\})} \tilde\pi^*f^*\b \wedge f^*\tau) \right)
   d \mu_{K(U_i)} (y)\\
 = &\, \sum_i \int_{K(U_i)} \left( \int_{S_i'\cap (D^{k-q}\x \{y\})} f^*\b \right)
   \left( \int_{f(D^q)} \tau \right) d \mu_{K(U_i)} (y)\\
 = &\, \sum_i \int_{K(U_i')} \left( \int_{S_i'\cap (D^{k-q}\x \{y\})} f^*\b \right)
   d \mu_{K(U_i')} (y)\\
 = &\, \la [f',S'_\mu], \b\ra\, .
 \end{aligned}
 $$
\end{proof}

\begin{theorem} \label{thm:product-dual-2}
 Suppose that $f_1:S_{1,\mu_1}\to M$, $f_2:S_{2,\mu_2}\to M$ are two
oriented measured immersed solenoids in $M$
 intersecting transversally, and let $f,S_\mu\to M$ be its intersection solenoid. Then
 the duals of the Ruelle-Sullivan homology classes satisfy
  $$
  [f,S_\mu]^* = [f_1,S_{1,\mu_1}]^* \cup [f_2,S_{2,\mu_2}]^* \, .
  $$
\end{theorem}

\begin{proof}
  Note that $[f_1,S_{1,\mu_1}]^* \in H^{n-k_1}_c(M)$
  and $[f_2,S_{2,\mu_2}]^* \in H^{n-k_2}_c(M)$, so
  $[f_1,S_{1,\mu_1}]^* \cup [f_2,S_{2,\mu_2}]^*$ and $[f,S]^*$ both live in
  $$
  H^{n-k_1+n-k_2}_c(M) =H^{n-k}_c(M)\, .
  $$

 Consider the immersed solenoid $(F, S_1\x
  S_2)$, where $F=f_1\x f_2:S_1\x S_2\to M\x M$ and $S_1\x S_2$
  has the transversal measure $\mu$ given by
  (\ref{eqn:mu-product}). Let us see that the following equality,
  involving the respective generalized currents,
  $$
  [F,(S_1\x S_2)_\mu]=
  [f_1,S_{1,\mu_1}] \ox [f_2,S_{2,\mu_2}] \in H_{k_1+k_2}(M\x M)
  $$
 holds. We prove this by
 applying both sides to $(k_1+k_2)$-cohomology
classes in $M\x M$. Using the K\"unneth decomposition it is enough
to evaluate on a form $\b=p_1^*\beta_1\wedge p_2^*\b_2$, where
$\b_1,\b_2\in H^*(M)$ are closed forms and $p_1,p_2:M\x M\to M$ are
the two projections. Let $\{U_i\}$, $\{V_j\}$ be open covers of
$S_1$, $S_2$ respectively, by flow-boxes, and let $\{\rho_{1,i}\}$,
$\{\rho_{2,j}\}$ be partitions of unity subortinated to such covers.
Then
 $$
 \begin{aligned}
 \la & [F,(S_1\x S_2)_\mu],\b\ra  = \\ &= \sum_{i,j} \int_{K(U_i)\x K(V_j)}
 \left( \int_{L_{y}\x L_z} (p_1^*\rho_{1,i})\,
 (p_2^*\rho_{2,j}) F^*(p_1^*\beta_1\wedge
 p_2^*\b_2) \right) d\mu_{K(U_i)\x K(V_j)} (y,z) \\ &\ = \sum_{i,j} \int_{K(U_i)\x K(V_j)}
 \left( \int_{L_{y}\x L_z} p_1^*(\rho_{1,i} f_1^*\beta_1) \wedge p_2^*( \rho_{2,j}
 f_2^*\b_2) \right) d\mu_{1,K(U_i)}(y) \, d\mu_{2,K(V_j)}(z) \\ &\ =\left( \sum_{i} \int_{K(U_i)}
 \left( \int_{L_{y}} \rho_{1,i} f_1^*\beta_1 \right) d\mu_{1,K(U_i)} (y) \right)
 \left( \sum_{j} \int_{K(V_j)}
 \left( \int_{L_{z}} \rho_{1,j} f_2^*\beta_2 \right) d\mu_{2,K(V_j)} (y) \right) \\
 &\ =\la [f_1,S_{1,\mu_1}],\b_1\ra \, \la [f_2,S_{2,\mu_2}],\b_2\ra\, ,
 \end{aligned}
 $$
as required.

Now we are ready to prove the statement of the theorem. Let
$\varphi:M\to \Delta$ be the natural diffeomorphism of $M$ with the
diagonal $\Delta \subset M\x M$, and let $i:\Delta \inc M\x M$ be
the inclusion. Then, using theorem \ref{thm:product-dual-2},
  $$
  \begin{aligned} \
  [f,S_\mu]^* \, &=  [\varphi \circ f, S_\mu]^* = i^* ([F,(S_1\x
  S_2)_\mu]^*)= \\
   &= i^* ([f_1,S_{1,\mu_1}]^* \ox [f_2,S_{2,\mu_2}]^* )
  = [f_1,S_{1,\mu_1}]^* \cup [f_2,S_{2,\mu_2}]^* \, .
  \end{aligned}
  $$
\end{proof}

\medskip

Let us look more closely to the case where $k_1+k_2=n$.
We assume that $f_1:S_{1,\mu_1}\to M$, $f_2:S_{2,\mu_2}\to M$
are two oriented immersed measured solenoids of
dimensions $k_1,k_2$ respectively,
which intersect transversally. 
Let $f:S_\mu\in M$ be the intersection $0$-solenoid of
$f_1:S_{1,\mu_1}\to M$ and $f_2:S_{2,\mu_2}\to M$.

\begin{definition}\textbf{\em (Intersection index)}
\label{def:intersection-index}
At each point $x=(x_1,x_2)\in S$, 
the intersection index $\epsilon(x_1,x_2)\in \{\pm 1\}$
is the sign of the intersection of the leaf of $S_1$ through
$x_1$ with the leaf of $S_2$ through $x_2$. 
The continuous function $\epsilon:S\to
\{\pm 1\}$ gives the orientation of $S$.
\end{definition}

Recall that the $0$-solenoid $f:S_\mu\to M$ comes equipped with a natural measure $\mu$
(for a $0$-solenoid the notions of measure and transversal measure coincide).
If $x=(x_1,x_2)\in S$, then locally around $x$, $S$ is homeomorphic to
$T=T_1\x T_2$, where $T_1$ and $T_2$ are small local transversals of $S_1$
and $S_2$ at $x_1$ and $x_2$, respectively. The measure $\mu_T$ is the
product measure  $\mu_{1,T_1} \x \mu_{2,T_2}$.

\begin{definition}\textbf{\em (Intersection measure)}\label{def:intersection-measure}
The intersection measure is the transversal measure $\mu$ of the
intersection solenoid $f:S_\mu\to M$, induced by those of
$f_1:S_{1,\mu_1}\to M$ and $f_2:S_{2,\mu_2}\to M$.

\end{definition}

\begin{definition}\textbf{\em (Intersection pairing)} \label{def:intersection-pairing}
We define the intersection pairing as the real number
 $$
 (f_1, S_{1,\mu_1}) \cdot (f_2,S_{2,\mu_2})=
 \int_{S } \epsilon  \ d \mu \, .
 $$
\end{definition}

\begin{theorem} \label{thm:product-RS}
If $f_1:S_{1,\mu_1}\to M$ and $f_2:S_{2,\mu_2}\to M$ are two oriented immersed measured solenoids of
dimensions $k_1,k_2$ respectively,
which intersect transversally, such that $k_1+k_2=n$. Then
 $$
 (f_1, S_{1,\mu_1}) \cdot (f_2,S_{2,\mu_2})=
 [f_1, S_{1,\mu_1}]^* \cdot [f_2,S_{2,\mu_2}]^*\, .
 $$
\end{theorem}

\begin{proof}
By theorem \ref{thm:product-dual-2},
 $$
 [f_1, S_{1,\mu_1}]^* \cup [f_2,S_{2,\mu_2}]^* =
 [f,S_\mu]^* \in H^n_c(M,\RR)\, .
 $$
The intersection product $ [f_1, S_{1,\mu_1}]^* \cdot
[f_2,S_{2,\mu_2}]^*$ is obtained by evaluating this cup product on
the element $1\in H^0 (M,\RR)$, i.e.
  $$
  [f_1, S_{1,\mu_1}]^* \cdot [f_2,S_{2,\mu_2}]^* =
  \la [f,S_\mu],1\ra = \int_{S} f^*(1) d\mu(x) =
  \int_{S} \epsilon \, d\mu \, ,
  $$
since the pull-back of a function gets multiplied by the orientation
of $S$, which is the function $\epsilon$.
\end{proof}

\medskip

When the solenoids are uniquely ergodic we can sometimes recover this
intersection index by a natural limiting procedure.
Recall that we say that a Riemannian solenoid $S$ is of controlled growth
(see definition 3.3 in \cite{MPM2}) if there is a leaf $l\subset S$, a
point $p\in l$ such that the Riemannian balls $l_n\subset l$,
of some radius $R_n\to \infty$, satisfy that for each flow-box $U$ in a
finite covering of $S$ the number of incomplete horizontal discs in $U\cap l_n$
is negligeable with respect to the number of complete horizontal discs in
$U\cap l_n$. Then, if $\mu_n$ is the normalized measure corresponding to $l_n$,
the limit $\mu=\lim_{n\to\infty} \mu_n$ is the unique Schwartzman measure
(corollary 3.7 in \cite{MPM2}).

\begin{theorem}
Let $f_1:S_{1,\mu_1}\to M$ and $f_2:S_{2,\mu_2}\to M$ be two immersed, oriented,
uniquely ergodic solenoids with controlled growth transversally intersecting.
Let $l_1\subset S_{1}$ and $l_2\subset
S_{2}$ be two arbitrary leaves. Choose two base points $x_1\in
l_1$ and $x_2 \in l_2$, and fix Riemannian exhaustions $(l_{1,n})$
and $(l_{2,n})$. Define
 $$
 (f_1,l_{1,n}) \cdot (f_2,l_{2,n})=\frac{1}{M_n} \sum_{p=(p_1,p_2)\in
 l_{1,n}\times l_{2,n} \atop f_1(p_1)=f_2(p_2)} \epsilon (p)\, ,
 $$
where $M_n =\Vol_{k_1}(l_{1,n})\cdot \Vol_{k_2}(l_{2,n})$.

Then
 $$
 \lim_{n\to +\infty }
 (f_1,l_{1,n}) \cdot (f_2,l_{2,n})= (f_1,S_{1\mu_1}) \cdot (f_2,S_{2,\mu_2}) \, .
 $$
In particular, the limit exists and is independent of the choices of $l_1$,
$l_2$, $x_1$, $x_2$ and the radius of the Riemannian exhaustions.
\end{theorem}

\begin{proof}
The key observation is that because of the unique ergodicity, the
atomic transversal measures associated to the normalizad $k$-volume
of the Riemannian exhaustions (name them $\mu_{1,n}$ and
$\mu_{2,n}$) are converging to $\mu_1$ and $\mu_2$, respectively.
In particular, in each local flow-box we have
 $$
 \mu_{1,n}\times \mu_{2,n}\to \mu_1\times \mu_2 =\mu \, .
 $$
Therefore the average defining $(f_1,l_{1,n})\cdot (f_2,l_{2,n})$
converges to the integral defining
$(f_1,S_{1\mu_1})\cdot (f_2,S_{2,\mu_2})$ since $\epsilon$ is a
continuous and integrable function (indeed bounded by $1$).
\end{proof}

\begin{remark}
The previous theorem and proof work in the same form for
ergodic solenoids, provided that we know that the Schwartzman limit
measure for almost all leaves is the given ergodic measure. This is
simple to prove for ergodic solenoids with trapping regions mapping to a
contractible ball in $M$
(cf.\ theorem 7.12 in \cite{MPM2}).
\end{remark}


\bigskip

We end up this section with a perturbation result.
We want to prove that we can achieve transversality for a
large class of solenoids by a suitable homotopy. The solenoids that we
have in mind are those whose transversal is a Cantor set.

\begin{theorem} \label{thm:perturbing}
  Let $f:S\to M$ be a solenoid whose transversals are Cantor sets, and let $N\subset M$
be a smooth closed submanifold. Then we can homotop $f:S\to M$ so that $N$ and $S$
intersect transversely.
\end{theorem}

\begin{proof}
  Let $f:S\to M$ be an immersion of the solenoid into a manifold $M$. Recall that
this means that the diferential of $f$ along leaves is injective.

   Let $p$ be a point in the solenoid. We want to perturb $f$ in a neighbourhood of $p$.
Consider a flow-box $U\subset S$ of the form $D^k_{1+r}\x T$, where $T$ is a Cantor set,
and $D^k_{1+r}$ is a $k$-disc of radius $1+r$, for some small
real number $r>0$.
Consider also a coordinate chart $(x_1,\ldots, x_n)$
for $M$ so that $N$ is given by $x_1=\ldots=x_q=0$, and
$p=(0,y_0)$. Define the composition
 $$
 f: D^k_{1+r}\x T \to V \to \RR^p\, .
 $$

The transversality of the leaf $L_y=D^k_1 \x \{y\}$ to $N$ is equivalent to the
transversality of the map $f_y:=f(\,\cdot\, ,y)$ to zero.

For any $\epsilon>0$ small enough, there is a vector $v\in \RR^p$ so that
$f_{y_0}$ is transversal to $-v$. Therefore $f_{y_0}+v$ is transversal to zero
(that is, if $f_{y_0}(x)+v=0$ then $df_{y_0}(x)$ is surjective). Moreover,
there is an open neighbourhood of $y_0$, $T_0\subset T$, where this transversality still holds.

Now take a bump function $\rho(x)$ which is one over $D_1^k$ and it is zero near the boundary
of $D_{1+r}^k$. The map
 $$
 \hat{f}(x,y)= \left\{ \begin{array}{ll} f(x,y) + \rho(x)v, \qquad & x\in D_{1+r}^k, y\in T_0 \\
                        f(x,y), & \text{otherwise}
                       \end{array} \right.
 $$
is smooth (here it is where we use that $T$ is a Cantor set), transversal to $N$ along $D^k_1\x T_0$.

Repeating this process, we can find a finite cover $T=\sqcup \, T_j$, and define the perturbations
indepedently on $D_{1+r}^k\x T_j$.

\bigskip

Finally, we have managed to achieve transversality on $V=D^k_1\x T$, perturbing on $D_{1+r}^k\x T$.
What we do now is to use a finite cover of $S$ with subsets as $V$, and perturb successively.
At each step we take a perturbation of norm small enough so that
this does not destroy the perturbation over the set where it was previously achieved.

\end{proof}

\begin{remark} \label{rem:prob-}
  We can construct an example where it is not possible to perturb
  two solenoids (at least in a differentiable way) with transversal 
  Cantor sets so that they intersect transversally.

  Consider $M=\RR^2$, and let $K_1,K_2\subset [0,1]$ be two Cantor sets.
  Let $S_1$ be given by the leaves $(x,y)$, $x\in \RR$, $y\in K_1$.
  Let $S_2$ be given by the leaves $(x,x^2+ z)$, $x\in \RR$, $z\in K_2$.
 These two solenoids intersect non-transversally at the points determined by
 $x=0$, $y=z \in K_1\cap K_2$.

 Suppose that we have small perturbations $S_1',S_2'$ of $S_1,S_2$, respectively.
 Then $S_1'$ is defined by leaves of the form $(x, y+f_1(x,y))$, $x\in \RR$, $y\in K_1$, and $S_2'$ by leaves
 of the form $(x, x^2+z + f_2(x,z))$, $x\in \RR$, $z\in K_2$, where $f_j$ is a smooth function on $\RR \x
 K_j$, having small norm, $j=1,2$ (recall that a smooth function on $\RR\x K_j$ extends as a smooth function on
 some neighbourhood of it). Composing with a suitable diffeomorphism of $\RR^2$, we can
 suppose that $f_1=0$. So we are looking for non-transversal intersections of
 $S_1'=\{(x,y)\, | \, x\in \RR,  y\in K_1\}$ and
 $S_2'=\{(x,x^2+z+g(x,z))\, | \, x\in \RR,  z\in K_2\}$, $g$ some small smooth function on $\RR \x K_2$.
 These are obtained by solving 
  $$
  2x+ g_x(x,z)=0 , \qquad y=x^2+z+g(x,z)\, .
  $$ 
 The equation $2x+ g_x(x,z)=0$
 can be solved as $x=\phi(z)$, for some (small) smooth function $\phi$.
Write $r(z)=\phi(z)^2 +z + g(\phi(z),z)$, which is a smooth function on $K_2$
 close to $r_0(z)=z$. This defines a smooth isotopy of $K_2$. Let $K_2'=r(K_2)$.
The points of non-transversal intersections of $S_1'$ and $S_2'$ are 
given by solving $y=r(z)$, so they correspond to the points in $K_1\cap K_2'$.

To guarantee that $K_1\cap K_2'\neq \emptyset$, just choose $K_1$, $K_2$ two Cantor sets with
positive Lebesgue measure such that $\mu(K_1)+\mu(K_2)>1$. As $K_2'$ is a small smooth perturbation
of $K_2$, the measure of $K_2'$ is close to that of $K_2$. So $K_1$ and $K_2'$ must intersect.
\end{remark}

\section{Almost everywhere transversality} \label{sec:aet}

The intersection theory developed in section \ref{sec:intersection}
is not fully satisfactory since there are examples of solenoids
(e.g.\ foliations) which do not intersect transversally, and cannot
be perturbed to do so. 
Even in the case of solenoids whose transversal structure is Cantor, sometimes
it is not possible to perturb smoothly the solenoid to make them intersect transversally, as 
remark \ref{rem:prob-} shows. Another example is given by 
the persistence of homoclinic tangencies for stable and unstable foliations shows (see \cite{Takens}).

However, a weaker notion is enough to
develop intersection theory for solenoids. Indeed, the intersection
pairing can also be defined for oriented, measured solenoids
$f_1:S_{1,\mu_1}\to M$ and $f_2:S_{2,\mu_2}\to M$, immersed in an oriented
$n$-manifold $M$, with $k_1+k_2=n$, $k_1=\dim S_1$, $k_2=\dim S_2$,
which intersect transversally almost everywhere in the following
sense:

\begin{definition}\textbf{\em (Almost everywhere transversality)}\label{def:almost-transverse}
Let $f_1:S_{1,\mu_1}\to M$ and $f_2:S_{2,\mu_2}\to M$ be two measured
immersed oriented solenoids. They intersect almost everywhere
transversally if
the set
 $$
 \begin{aligned}
 F=\{(p_1,p_2) & \,\in S_1\x S_2 \, ; \\ &\,  f_1(p_1)=f_2(p_2),
 df_1(p_1) (T_{p_1}S_1)+df_2(p_2) (T_{p_2}S_2) \neq T_{f_1(p_1)}M \}
 \subset S_1\x S_2
  \end{aligned}
 $$
of non-transversal intersection points satisfies:
\begin{enumerate}
\item[(1)] every point $p\in F$ is an isolated point of
 $$
 S= \{(p_1,p_2) \in S_1\x S_2 \, ; \, f_1(p_1)=f_2(p_2) \}
 \subset S_1\x S_2
 $$
 in the leaf of $S_1\x S_2$ through $p$.
\item[(2)] $F$ is null-transverse in
$S_1\x S_2$ (with the natural product transversal measure $\mu$), i.e.
if the set of leaves of $S_1\x S_2$ intersecting $F$ has zero
$\mu$-measure.
\end{enumerate}
\end{definition}

It is useful to translate to $S_1\times S_2$ the meaning of
almost everwywhere transversality.

\begin{definition}\textbf{\em (Almost everywhere transversality)}\label{def:almost-transverse-2}
Let $f: S_{\mu}\to M$ be a measured
immersed oriented $k$-solenoid and $N\subset M$ a closed submanifold of codimension $k$.
They intersect almost everywhere
transversally if the set
 $$
 F=\{p \in S \, ;  \, f (p)\in N, \,
 df(p) (T_{p}S)+ T_{f(p)}N \neq T_{f(p)} M \}
 \subset S
 $$
of non-transversal intersection points satisfies that:
 \begin{itemize}
 \item[(1)] every point $p\in F$ is isolated as a point of $S'$ in the leaf of $S$ through $p$,
 \item[(2)] $F\subset S_\mu$ is null-transverse, i.e.\
 for any flow-box $U=D^k\x K(U)$, the projection by $\pi:U=D^k\x K(U) \to K(U)$
 of the intersection $F\cap U$, that is $\pi (F\cap U)\subset K(U)$, is
 of zero $\mu_{K(U)}$-measure in $K(U)$.
 \end{itemize}
\end{definition}

Note that a set $F\subset S_\mu$ in a measured solenoid
is null-transverse if for any local
transversal $T$, the set of leaves passing through $F$
intersects $T$ in a set of zero $\mu_T$-measure.

Every point of $S'-F$ is automatically isolated as a point of $S'$ in the leaf of $S$ through
it. Therefore condition (1) is equivalent to saying that every point of $S'$ is
isolated in the corresponding leaf.

\medskip

Then we have the following
straightforward lemma.

\begin{lemma}
The solenoids $f_1: S_{1,\mu_1}\to M$ and $f_2:S_{2,\mu_2}\to M$ are almost
everywhere transversal if and only if $f_1\x f_2: (S_1\x S_2)_\mu\to M\x M$ and the
diagonal $\Delta \subset M\x M$ intersect almost everywhere
transversally.
%
\end{lemma}

\medskip

Let $f:S_\mu\to M$ be an immersed solenoid intersecting transversally
almost everywhere a closed submanifold $N\subset M$. Write
$S'=f^{-1}(N)$ and let $F\subset S'$ be the subset of
non-transversal points. Note that $S'_{reg}=S'-F$ is open in $S'$
and $F$ is closed. Moreover, $S'_{reg}$ consists of the transversal
intersections, so the intersection index $\epsilon:S'_{reg} \to
\{\pm 1\}$ is well defined and continuous. We define the
intersection number as
 $$
 \int_{S'-F} \epsilon(x) d\mu(x) \, .
 $$

\begin{theorem} \label{thm:intersect-number-almost}
Suppose that  an immersed measured oriented
$k$-solenoid $f:S_\mu\to M$ and a submanifold $N\subset
M$ of codimension $k$ intersect almost everywhere transversally.
Then
 $$
 [f, S_\mu]^* \cdot [N] = \int_{S'-F} \epsilon \ d\mu \, .
 $$
\end{theorem}

\begin{proof}
Fix an accessory Riemannian metric on $M$. By pull-back, this gives a metric
on $S$.

Let $p\in F$. Then by assumption, there is some $\eta>0$ such that
 $$
 B_{\frac32 \eta}(p)\cap F=B_{\frac32 \eta}(p)\cap S'=\{p\},
 $$
where $B_{r}(p)$ is the
Riemannian ball in the leaf centered at $p$ and of radius $r>0$.
It is easy to construct a flow-box
$U=D^k\x K(U)$ with $p=(0,y_0)\in U$ so that
\begin{itemize}
 \item[(i)] $D^k\x\{y_0\}=B_{\frac32 \eta} (p)$,
 \item[(ii)] $D^k_{3/4} \x\{y_0\}=B_{\eta} (p)$, ($D^k_r$ denotes the open disc of radius $r>0$),
 \item[(iii)] the open annulus $A= (D^k - \bar D^k_{1/2})$ satisfies that
  $(A \x K(U) )\cap S'=\emptyset$,
 \item[(iv)] the intersection number
  $[f(D^k_{3/4} \x\{y \})]\cdot [N]$
  is constant for $y\in K(U)$.
\end{itemize}

For achieving this, take $K(U)$ small enough. Note that the
intersection number in (iv) is well-defined
since $f(\bd (D^k_{3/4} )\x\{y \})$ does not touch $N$;
and it is locally constant by continuity.
We fix a finite covering $\{U_i\}$ of $F$ with such flow-boxes.

Let $\pi_i: U_i=D^k\x K(U_i) \to K(U_i)$ be the projection onto the second factor. By hypothesis,
$\pi_i(F\cap U_i)$ is of zero measure. We may take a nested
sequence $(V_{i,n})$ of open neighbourhoods of $\pi_i(F\cap U_i)$ in $K(U_i)$ such that
$\bigcap_{n\geq 1} V_{i,n}=\pi_i(F\cap U_i)$. Let
 $$
 U_{i,n}=D^k_{3/4}\x V_{i,n}
 $$
and
 $$
 U_n=\bigcup_i U_{i,n}\,.
 $$
Then $(U_n)$ is a nested
sequence of open neighbourhoods of $F$ in $S$. 
It may happen that  $\bigcap_{n\geq 1} U_n$ contains points of $S'-F$, but this is a set
of $\mu$-measure zero. So
 $$
 \int_{S'-U_n} \epsilon \ d\mu \too  \int_{S'-\cap_{n\geq 1} U_n} \epsilon \ d\mu=
 \int_{S'-F} \epsilon \ d\mu \, .
 $$

As $S'-U_n$ is compact,
the angle of intersection in $S'-U_n$ between $f(S)$ and $N$ is
bounded below, so there is a small $\rho>0$ (depending on $n$)
such that if $U_\rho$ is
the $\rho$-tubular neighbourhood of $N$ in $M$, then
for each
intersection point $x\in S'-U_n$, there is a (topological) disc
$D_x$ contained in a local leaf through $x$, which is exactly the
path component of $f^{-1}(U_\rho)$ through $x$. Making $\rho$ smaller
we can assume that $D_x$ is as small as we want.
Note that (iii) guarantees that $D_x$ does not touch
$D^k_{3/4} \x V_{i,n}=U_{i,n}$ for any $i$. So $D_x\subset S-U_n$.

Let $\tau_\rho$ be a Thom form for $N\subset M$, that is a closed
$k$-form supported in $U_\rho$, whose integral in the
normal space to $N$ is one. Then $\int_{D_x} \tau_\rho=1$ for any $x\in S'-U_n$.
So
 $$
 \int_{S'-U_n} \epsilon \ d\mu = \int_{S-U_n}  f^*\tau_\rho\, .
 $$

On the other hand,
 $$
 \int_{U_{i,n}}  f^*\tau_\rho =\int_{V_{i,n}}
 \left(\int_{D^k_{3/4} \x\{y\}} f^*\tau_\rho \right) d\mu_{K(U_i)} \leq
 C \, \mu_{K(U_i)}(V_{i,n}) \to 0\,,
 $$
where $C$ is a bound for all the intersection numbers in (iv)
for all $U_i$ simultaneously.
Then
 $$
 \int_{U_n}  f^*\tau_\rho \to 0\, ,
 $$
when $n\to \infty$.

Putting everything together,
 $$
 \begin{aligned} \,
  [f, S_\mu]^* \cdot [N] \, &= \la  [f, S_\mu] , [\tau_\rho]\ra  = \int_{S_\mu}
 f^*\tau_\rho = \int_{S-U_n} f^*\tau_\rho  +\int_{U_n} f^*\tau_\rho  \\ &=
 \int_{S'-U_n} \epsilon \ d\mu +
 \int_{U_n} f^*\tau_\rho \to  \int_{S'-F} \epsilon \ d\mu
 \, .
 \end{aligned}
 $$
\end{proof}

\begin{remark}
It is not true that, without further restrictions,
the measure $\mu$ is finite on $S'_{reg}=S'-F$. For instance, it may happen that
around a point $p\in F$, there are leaves of $S$ (leaves which do not go through $F$)
with arbitrary large number of positive and negative intersections with $N$ (near $p$).
Obviously, the difference between positive and negative intersections is bounded.
Therefore, there is no current associated to $(S'_{reg},\mu)$.
\end{remark}

\bigskip

Consider now two immersed measured oriented solenoids $f_1:
S_{1,\mu_1}\to M$, $f_2:S_{2,\mu_2}\to M$ intersecting almost everywhere
transversally. Let $F\subset S_1\x S_2$ be the subspace of
non-transversal intersection points, which has null-transversal
measure in $S_1\x S_2$. Set $S= (S_1 \x S_2) \cap \tilde{f}^{-1}(\Delta)$,
where $\tilde{f}=f_1\x f_2$.
Then there is an intersection index $\epsilon(x)$ for each $x\in S
-F$ and an intersection measure $\mu$ on $S-F$. We define the
intersection product as
 $$
 \int_{S-F} \epsilon  \ d \mu  \, .
 $$
Then theorem \ref{thm:intersect-number-almost} implies the following:

\begin{theorem}
In the situation above, we have that
 $$
 [f_1, S_{1,\mu_1}]^* \cdot [f_2,S_{2,\mu_2}]^* =
\int_{S-F} \epsilon (x) \ d \mu(x) \, .
 $$
\end{theorem}

%
%
%
%
%
%
%
%
%
%

\section{Intersection of analytic solenoids} \label{sec:analytic}

It is now our intention to translate the theory of solenoids intersecting
almost-everywhere transversally to the case where the
dimensions are not complementary, that is, when $k_1+k_2>n$.

\begin{definition}\textbf{\em (Almost everywhere transversality)}\label{def:almost-transverse-3}
Let $f_1: S_{1,\mu_1}\to M$ and $f_2:S_{2,\mu_2}\to M$ be two measured
immersed oriented solenoids in an oriented $n$-manifold $M$, with
$k_1+k_2 \geq n$, $k_1=\dim S_1$, $k_2=\dim S_2$. They intersect almost everywhere
transversally if
the set
 $$
 \begin{aligned}
 F=\{(p_1,p_2) & \,\in S_1\x S_2 \, ; \\ &\,  f_1(p_1)=f_2(p_2),
 df_1(p_1) (T_{p_1}S_1)+df_2(p_2) (T_{p_2}S_2) \neq T_{f_1(p_1)}M \}
 \subset S_1\x S_2
  \end{aligned}
 $$
of non-transversal intersection points satisfies:
\begin{enumerate}
\item[(1)] every point $p\in F$ is an isolated point of $F$
 in the leaf of $S_1\x S_2$ through $p$.
 \item[(2)] the set
  $$
  S =\{(p_1,p_2) \in S_1\x S_2 \, ; \,  f_1(p_1)=f_2(p_2) \}
 \subset S_1\x S_2
  $$
  is $C^{1,0}$-conjugate, locally near any $p\in F$, to a leafwise (real) analytic
  set (i.e. it is of class $C^{\omega,0}$).
\item[(3)] $F$ is null-transverse in
$S_1\x S_2$ (with the natural product transversal measure $\mu$), i.e.
if the set of leaves of $S_1\x S_2$ intersecting $F$ has zero
$\mu$-measure.
\end{enumerate}
\end{definition}

Then it is useful to translate to $S_1\times S_2$ the meaning of
almost everwywhere transversality.

\begin{definition}\textbf{\em (Almost everywhere transversality)}\label{def:almost-transverse-4}
Let $f: S_{\mu}\to M$ be a measured
immersed oriented $k$-solenoid and $N\subset M$ a closed submanifold of codimension $q$.
They intersect almost everywhere
transversally if the set
 $$
 F=\{p \in S \, ;  \, f (p)\in N, \,
 df(p) (T_{p}S)+ T_{f(p)}N \neq T_{f(p)} M \}
 \subset S
 $$
of non-transversal intersection points satisfies that:
 \begin{itemize}
 \item[(1)] every point $p\in F$ is isolated as a point of $F$ in the leaf of $S$ through $p$,
  \item[(2)] the set
  $$
  S' =\{p \in S \, ; \,  f(p)\in N\} \subset S
  $$
  is $C^{1,0}$-conjugate, locally near any $p\in F$, to a leafwise (real) analytic
  set (i.e. it is of class $C^{\omega,0}$).
 \item[(2)] $F\subset S_\mu$ is null-transverse, i.e.\
 for any flow-box $U=D^k\x K(U)$, the projection by $\pi:U=D^k\x K(U) \to K(U)$
 of the intersection $F\cap U$, that is $\pi (F\cap U)\subset K(U)$, is
 of zero $\mu_{K(U)}$-measure in $K(U)$.
 \end{itemize}
\end{definition}

We have the following
straightforward lemma.

\begin{lemma}
The solenoids $f_1: S_{1,\mu_1}\to M$ and $f_2:S_{2,\mu_2}\to M$ are almost
everywhere transversal if and only if $\tilde{f}=f_1\x f_2: (S_1\x S_2)_\mu
\to M\x M$ and the
diagonal $\Delta \subset M\x M$ intersect almost everywhere
transversally.
\end{lemma}

\medskip

Let $f:S_\mu\to M$ be an immersed oriented $k$-solenoid intersecting transversally
almost everywhere a closed oriented submanifold $N\subset M$ of codimension $q$. Write
$S'=f^{-1}(N)$ and let $F\subset S'$ be the subset of
non-transversal points. Note that $S'_{reg}=S'-F$ is open in $S'$
and $F$ is closed.
Let $p\in S'_{reg}=S'-F$. Then the transversal intersection property implies that there
exists a flow-box $U=D^k\x K(U)$ for $S$ around $p$, with coordinates
$(x_1,\ldots, x_k, y_1,\ldots, y_l)$ such that $N$ is defined by $x_1=\ldots
=x_q=0$. Thus $S'_{reg}$ has locally the
structure of $(k-q)$-dimensional oriented solenoid, with a transversal measure $\mu$
induced by $\mu$ and invariant by holonomy.
Note that $S'_{reg}$ is not a solenoid because it is not compact.

\begin{theorem} \label{thm:intersect-number-almost-2}
Suppose that  a $k$-solenoid $f:S_\mu\to M$ and a submanifold $N\subset
M$ of codimension $q$ intersect almost everywhere transversally.
Let $\omega \in \Omega^{k-q}(M)$ be a closed form. Then
 $$
 \la [f, S_\mu]^* \cup  [\omega] , [N]\ra  = \int_{S'-F} \omega \, .
 $$
\end{theorem}

\begin{proof}
Let us take a Thom form $\tau_\rho$ for $N$, and consider the current
 \begin{equation}\label{eqn:we}
 (f, S_\mu) \wedge \tau_\rho
 \end{equation}
defined as the wedge of the generalized current with the smooth form $\tau_\rho$.
Let us see that there is a limit for (\ref{eqn:we}) when $\rho \to 0$.

Let us define the current of integration $S'_{reg}$. This is obviously well-defined
off $F$. Now suppose that $\omega$ is a $(k-q)$-form supported in a small ball around
a point $p\in F$. Let $U=D^k\x K(U)$ be a flow-box around $p$, where $\omega$ is defined.
Then, after taking a $C^{1,0}$-diffeomorphism, we can suppose that $S'$ is defined
as $f(x,y)=0$, for some $f:D^k\x K(U) \to \RR^q$ of class $C^{\omega,0}$.
Then $S_y=S'_{reg}\cap L_y=f_y^{-1}(0)$ is an analytic subset. Hence the integral of $\omega$
on $S_y$ is bounded
 $$
 \int_{S_y} \omega \leq C ||\omega||\, ,
 $$
where $C$ is a constant that we can suppose valid for all $y\in K(U)$ by continuity.

Moreover, making the radius of the ball smaller, we have that $C\to 0$.
Hence
 $$
 \la (S_{reg}),\omega\ra := \int_{K(U)} \left( \int_{S_y} \omega \right) d\mu_{K(U)}(y)
 $$
is well-defined.

Now, to see that
 $$
 (f, S_\mu) \wedge \tau_\rho \to (S_{reg})\, ,
 $$
we apply both sides to a $(k-q)$-form $\omega$. For $\omega$ supported in a flow-box off $F$,
we have that
 $$
 \la (f, S_\mu) \wedge \tau_\rho ,\omega \ra  =\int_{K(U)} \int_{L_y}
 \tau_\rho\wedge\omega \, d\mu_{K(U)}(y)\to \int_{K(U)} \int_{S_y} \omega \, d\mu_{K(U)}(y)=
 \la (S_{reg}),\omega\ra \, .
 $$

Now let $\omega$ supported in an $\epsilon$-ball around $p\in F$. Then $\la (S_{reg}),\omega\ra$
is as small as we want, and
 $$
 \lim_{\rho\to 0} \la (f, S_\mu) \wedge \tau_\rho ,\omega \ra = \lim_{\rho\to 0} \int_{K(U)} \int_{L_y}
 \tau_\rho\wedge\omega
 $$
is small. Since $\mu_{K(U)}({K(U)})$ is small, it only remains to see that
 $$
 \lim_{\rho\to 0} \int_{L_y} \tau_\rho\wedge\omega  = \int_{S_y} \omega  
 $$
is bounded for $y$ off the bad locus. This is bounded by the area of $S_y$ times the norm
of $\omega$, and both these quantities are bounded (the first one
is bounded due to the transversal continuity).

%

 \end{proof}

Consider now two immersed measured oriented solenoids $f_1:
S_{1,\mu_1}\to M$, $f_2:S_{2,\mu_2} \to M$ intersecting almost everywhere
transversally. Let $F\subset S_1\x S_2$ be the subspace of
non-transversal intersection points, which has null-transversal
measure in $S_1\x S_2$. Set $S= (S_1 \x S_2) \cap \tilde{f}^{-1}(\Delta)$,
where $\tilde{f}=f_1\x f_2: S_1\x S_2\to M\x M$.
Then Theorem \ref{thm:intersect-number-almost-2} implies that
 $$
 \la [f_1, S_{1,\mu_1}]^* \cup [f_2,S_{2,\mu_2}]^*, [\omega]\ra
 = \int_{S-F} \omega \, ,
 $$
for any closed form $\omega$ of degree $k_1+k_2-n$.

\begin{corollary}
  Let $M$ be an analytic manifold, and let  $f_1:
S_{1,\mu_1}\to M$, $f_2:S_{2,\mu_2} \to M$ be two immersed measured oriented solenoids
of class $C^{\omega,0}$ (that is, with analytic leaves).
Let $S= \{(p_1,p_2)\in S_1\x S_2 \, ; \,  f_1(p_1)=f_2(p_2) \}$, and let
 $F \subset S$ consist of points $(p_1,p_2)$ such that the leaves of
$S_1$ and $S_2$ at $p_1$ and $p_2$ do not intersect transversally.
Suppose that
\begin{itemize}
\item[(1)] every point $p\in F$ is an isolated point of $F$
 in the leaf of $S_1\x S_2$ through $p$.
\item[(2)] $F$ is null-transverse in
$S_1\x S_2$.
\end{itemize}
Then
 $$
 \la [f_1, S_{1,\mu_1}]^* \cup [f_2,S_{2,\mu_2} ]^* , [\omega] \ra  = \int_{S-F} \omega \, ,
 $$
for any $[\omega]\in H^{k_1+k_2-n}(M)$.
\end{corollary}

\begin{proof}
 We only need to note that condition (2) in Definition \ref{def:almost-transverse-3} is automatic.
\end{proof}

\end{document}